\documentclass{amsart}
\usepackage{amsgen,amsmath,amstext,amsbsy,amsopn,amsfonts,amssymb}
\newtheorem{theorem}{Theorem}[section]
\newtheorem{lemma}[theorem]{Lemma}
\newtheorem{corollary}[theorem]{Corollary}
\newtheorem{proposition}[theorem]{Proposition}
\newtheorem{obs}[theorem]{Observation}

 \newtheorem{defi}[theorem]{Definition}
\newenvironment{definition}{\begin{defi}\rm}{\end{defi}}
\newtheorem{exa}[theorem]{Example}
\newenvironment{example}{\begin{exa}\rm}{\end{exa}}
\newtheorem{rem}[theorem]{Remark}
\newenvironment{remark}{\begin{rem}\rm}{\end{rem}}
\newtheorem{rems}[theorem]{Remarks}

\newtheorem{ack}[theorem]{Acknowlegment}

\def\bsq{\blacksquare\medskip}

\def\H{\mathcal H}

\def\N{\mathcal N}

\def\B{\mathcal B}

\def\V{\mathcal V}

\def\U{\mathcal U}
\def\R{\mathcal R}

\def\NN{{\mathbf N}}
\def\ZZ{{\mathbf  Z}}
\def\CCC{{\mathbf  C}}
\def\RRR{{\mathbf R}}

\def\RR+{{\mathbf R}^*}

\def\PP{\mathbf P}

\def\Q_p{{\mathbf Q}_p}
\def\OO{\mathbf O}
\def\UU{\mathbf U}

\def\eps{\varepsilon}
\def\Ga{\Gamma}
\def\ga{\gamma}
\def\la{\lambda}
\def\vfi{\varphi}

\def\ExAut{{\mathrm Aut}_{\rm e}}
\def\Aut{{\mathrm Aut}}
\def\Inn{{\mathrm Inn}}

\def\Ad{{\mathrm Ad}}

\def\Hom{{\mathrm Hom}}
\def\vfi{{\varphi}}
\def\tout{\quad \quad \text{for all}\quad }

\begin{document}

\title[Local rigidity for actions on non commutative $L_p$-spaces] {Local rigidity for actions of Kazhdan groups on non commutative $L_p$-spaces}
\author{Bachir Bekka}

\address{Bachir Bekka \\ IRMAR \\ UMR-CNRS 6625 Universit\'e de  Rennes 1\\
Campus Beaulieu\\ F-35042  Rennes Cedex\\
 France}
 
\email{bachir.bekka@univ-rennes1.fr}

\begin{abstract}
Let $\Ga$ be a discrete group  and
$\N$ a finite factor, and assume that both have Kazhdan's Property (T).
For $ p \in[1,+\infty), \, p\neq 2,$ 
 let $\pi: \Ga \to \OO(L_p(\N))$ be a homomorphism 
 to the  group  $\OO(L_p(\N))$ of linear bijective  isometries 
of the $L_p$-space of $\N.$  There are two actions $\pi^l$ and $\pi^r$
 of a finite index subgroup $\Ga^+$ of $\Ga$ by automorphisms of $\N$ associated to $\pi$
 and given by $\pi^l(g)x=(\pi(g) 1)^*\pi(g)(x)$ and $\pi^l(g)x=   \pi(g)(x) (\pi(g) 1)^*$
 for $g\in \Ga^+$ and $x\in \N.$
Assume that  $\pi^l$ and $\pi^r$ are ergodic.
We prove that $\pi$ is locally rigid, that is,
the orbit of $\pi$ under $\OO(L_p(\N))$ is open in
$\Hom (\Ga,\OO(L_p(\N)) ).$ As a corollary, we obtain that,
if moreover $\Ga$ is an ICC group, then the embedding
$g\mapsto \Ad (\la(g))$  is locally rigid in   $\OO(L_p(\N(\Ga))),$
where  $\N(\Ga)$ is the von Neumann algebra generated
by the left regular representation $\la$ of $\Ga$.

\end{abstract}
\maketitle

\section{Introduction}

Let  $\Ga$ be a discrete group and $G$  a topological group.
A group homomorphism $\pi_0:\Ga\to G$ 
is locally rigid if every sufficiently small
deformation of $\pi_0$ is trivial, in the sense that it is
given by conjugation by elements from $G.$ More precisely,
let $\Hom(\Ga, G)$ be the set of all homomorphisms  $\pi: \Gamma\to G$
endowed with the topology of pointwise convergence on $\Gamma$. 
The group $G$ acts on $\Hom(\Ga, G)$ by conjugation:
$${\rm Ad}(g)\pi (\ga)= g\pi(\ga) g^{-1} \tout  g\in G, \ga\in\Ga.$$
We say that $\pi_0$ is  \emph{locally rigid} 
if  its $G$-orbit in $\Hom(\Ga, G)$ is open.

Local rigidity  was proved for the embedding of a  cocompact lattice $\Ga$
in a semisimple real Lie group  $G$ by Calabi, Vesentini, Selberg, and Weil 
 (see Chapter VII in \cite{Raghunathan}). 

Groups with Kazhdan's Property (T) are defined by  a rigidity property of   their unitary group representations  and   play an important role in a large variety of subjects (for an account on Kazhdan's groups, see the monography \cite{BHV}).
It is natural  to  study  local rigidity for homomorphisms from such groups 
to  various topological groups $G.$   
As an example,  it was shown in  \cite[Theorem1]{Rapinchuk}
that,  if $\Gamma$ is a (discrete) Kazhdan group, then  
every  unitary representation $\Gamma\to U(n)$  is locally rigid in $GL_n(\CCC).$
In recent years, there has been an increasing interest in  local rigidity for homomorphisms 
with ``infinite dimensional" groups as targets.
For instance, a striking result in \cite{FisherMargulis} shows that every action of a  Kazhdan group by  isometries on a compact
Riemannian manifold is locally rigid in its group of diffeomorphisms.
For an overview   on local rigidity for
 actions of groups  on various manifolds, see \cite{Spatzier} and \cite{Fisher}.

In this paper, we study  local rigidity
for homomorphisms of discrete Kazhdan groups into 
the group of linear isometries of  non-commutative $L_p$-spaces, that is, the $L_p$-spaces
 associated to a von Neumann algebra.
Recently,   Property (T) has been studied  in the
 framework  of  group actions by  isometries on  Banach spaces 
 and more specifically on   $L_p(X, \mu)$ for a measure space $(X,\mu)$; see \cite{BFGM}. 
 Some of the results from  \cite{BFGM} were
 extended in \cite{Olivier} to non-commutative $L_p$-spaces. 

Recall that a von Neumann algebra $\N$  is  said to be finite
if there exists a faithful  normal finite trace  $\tau$ on $\N.$ 
Let $1\leq p <\infty$; the non-commutative $L_p$-space $L_p(\N)$
 is the completion of $\N$ with respect to the norm  defined   by $\Vert x\Vert_p= (\tau (|x|^{p}))^{1/p}$
for $x\in \N.$ For a  survey on these spaces, see \cite{Pisier-Xu}.

The von Neumann algebra $\N$ is a factor if the centre of $\N$ is reduced to the scalar operators.
When $\N$ is a finite factor, then either $\N$
is finite dimensional, in which case $\N$ is isomorphic to a matrix algebra
$M_n(\CCC)$ equipped  with the usual (normalized) trace, or $\N$
is a so-called type $II_1$ factor.

An important class of examples of type $II_1$ factors is given by 
ICC groups. Recall that the   group $\Ga$  is ICC
if its conjugacy classes, except  $\{e\}$, are infinite.
In this case,  the von Neumann algebra $\N(\Ga)$ of $\Ga$ is a (finite) factor.
Recall that $\N(\Ga)$ 
is the von Neumann algebra generated by the left regular representation
$\lambda$ of $\Ga$ on $\ell_2(\Ga);$ thus,
$\N(\Ga)$ is the closure for the strong operator topology
of the linear span of $\{\la(g)\ :\ g\in\Ga\}$
in the algebra $\B(\ell_2(\Ga)).$

A notion of Kazhdan's property (T) for von Neumann algebras
was defined in  \cite{CoJo} (see Section~\ref{SS:Kazhdan} below)
and  it was shown there that, for an ICC group $\Ga$, the factor $\N(\Ga)$ has Property (T) if and only if the group $\Ga$
has Kazhdan's property (T).

Let $1\leq p <\infty$  and $\N$ a finite factor. The orthogonal  group $\OO(L_p(\N))$ of $L_p(\N),$ that is, the  group
of bijective linear isometries of $L_p(\N)$, is a topological group when  endowed 
with the strong  operator topology (see Section~\ref{SS:Isometries} below).
Observe that every automorphism  or anti-automorphism $\theta$ of $\N$  extends to a unique isometry of $L_p(\N)$,
since $L_p(\N)$ contains $\N$ as dense subspace and since $\theta$ preserves the trace on $\N.$
In this way, we identify   the \emph{extended automorphism group} $\ExAut (\N),$ 
that is the group of automorphisms or anti-automorphisms of $\N,$ 
with a subgroup of $\OO(L_p(\N))$. 

Let $p\neq 2.$ 
 The group  $\OO(L_p(\N))$ for $p\neq 2$ was described in  \cite{Yeadon} (see  Theorem~\ref{Yeadon} below).
It follows from this description that
$\OO(L_p(\N))$ contains  a	subgroup  $\OO^+(L_p(\N))$
of  index at most 2 such that, for every $U$ in $\OO^+(L_p(\N)),$
 the mappings
$$
U^{l}:x\mapsto U(1)^*U(x) \, \,\,  \text{and}\, \,\,  U^{r}:x\mapsto  U(x)U(1)^*
$$ 
are automorphisms of $\N.$

If   $\pi: \Ga \to \OO(L_p(\N))$
  is  a group homomorphism,  we obtain in this way 
 two  actions   of  a subgroup $\Ga^+$   of  index at most 2  by automorphisms of $\N$,
 given  by  homomorphisms $$\pi^l: \Ga^+\to \Aut (\N) \,\, \text{and}\,\, \pi^r: \Ga^+\to \ExAut (\N);$$
  we call $\pi^l$ and $\pi^r$ the actions of $\Ga^+$  by automorphisms associated to $\pi$
  (see Section~\ref{SS:Rep}).
 Recall that an action $\theta : \Ga\to \Aut(\N)$ of a group $\Ga$ by automorphisms on a von Neumann algebra
 $\N$ is \emph{ergodic} if the fixed point algebra 
 $$\N^\Ga=\{ x\in \N : \theta_g(x)=x \quad \text{for all} \quad g\in \Ga\}$$
  consists only of the scalar multiples of $1.$

Here is our main result.
\begin{theorem}
\label{Theo1} 
Let $\Ga$ be a discrete group  and $\N$ a finite factor, and assume that both have Property (T). For $p\in [1, +\infty),\,  p\neq 2,$ let
  $\pi: \Ga \to \OO(L_p(\N))$ be  a homomorphism from $\Ga$ to the group of linear bijective isometries
  of $L_p(\N).$  Assume that
  the  associated actions $\pi^l$ and  $\pi^r$  of $\Ga^+$ by automorphisms on $\N$ are both ergodic.
 Then $\pi$ is locally rigid.
 \end{theorem}

Thus, there  exists a neighbourhood $\V$ of $\pi$ such that every $\rho$ in $\V$ is conjugate to $\pi$ by some $U$ in  $\OO(L_p(\N)).$
In fact, we will determine explicitly, in terms of 
Kazhdan pairs for $\Ga$  and $\N$,  such a  neighbourhood $\V$ for which $U$ can be chosen to be close
the identity (see Remark~\ref{Rem-Quantitive} below).

As we will see in Section~\ref{S:Comments},   the various assumptions made in the statement of Theorem~\ref{Theo1}
are necessary in one form or another. 

Since $\N(\Ga)$ is a factor when $\Ga$ is an ICC group, the action $g \mapsto \Ad(\la(g))$ of $\Ga$ 
by automorphisms of $\N(\Ga)$ is  ergodic; so,  the following  corollary is an immediate consequence of the previous theorem.
\begin{corollary}
\label{Cor1} 
Let $\Ga$ be an ICC group with  Kazhdan's Property (T).
The embedding $g \mapsto \Ad(\la(g))$  of  $\Ga$ in  $\OO(L_p(\N(\Ga)))$ 
is locally rigid, for $p\in[1, +\infty),\,  p\neq 2.$ 
\end{corollary}

\begin{remark}
\label{Rem-LocRigid}
There is another natural action of  $\Ga$ by  isometries on $L_p(\N(\Ga))$: it is 
given  by the embedding  
$$\pi_0: \Ga \to \OO(L_p(\N(\Ga))), g \mapsto \la(g),$$  where  $ \la(g)$ denotes
the extension from $\N(\Ga)$  to $L_p(\N(\Ga))$ of  multiplication from the left by the unitary $\la(g)$.
As will be seen below (Example \ref{Exa2}),  $\pi_0$ 
may fail to be locally rigid when $\Ga$ is an ICC  group with Property (T). 
In contrast, it can be shown that, if we view
$\pi_0$ as a homomorphism  in the unitary group $\UU(\N(\Ga))$ of $\N(\Ga),$
then $\pi_0$ is locally rigid.
\end{remark}

Apart from Yeadon's description of the group of 
 isometries of $L_p(\N)$ for $p\neq 2,$
 the  proof of Theorem~\ref{Theo1} 
depends on the following three ingredients:   the first one (see Proposition~\ref{Pro-Mazur} below) is that
$\OO(L_p(\N))$ is isomorphic, as \emph{topological} group, to an appropriate subgroup of the 
group of isometries of the Hilbert space $L_2(\N);$  the second ingredient
is the fact  (see \cite{CoJo}) that the group of  outer automorphisms of a factor
with Property (T) is  discrete; the third ingredient is an extension of  a result from \cite{Kawa} 
showing  that certain   $1$-cohomology classes  
associated to actions of a Kazhdan group by automorphisms of  a finite factor
are open  (see Proposition~\ref{Pro-Kawa}). 
For this,  we use in a crucial way a rigidity property of  projective unitary representations
of Kazhdan groups from \cite{NPS}.

This paper is organized as follows. In Section~\ref{S:Preliminaries},
we collect the ingredients necessary for the proof of our main result.
 Section~\ref{S:ProofTheo1} is devoted to the proof of Theorem~\ref{Theo1}. 
In Section~\ref{S:Comments},  we present counter-examples in relation with  the various assumptions 
made in the statement of Theorem~\ref{Theo1}.

\section{Preliminaries}
\label{S:Preliminaries}
Let $\N$ be a \emph{finite factor}, fixed throughout this section. 

\subsection{The group of isometries of $L_p(\N)$}
\label{SS:Isometries}

The following result is a corollary of Yeadon's description  of the linear (not necessarily surjective)  isometries of the non-commutative $L_p$ space of a 
semi-finite von Neumann  algebra for $p\neq 2$ (see \cite[Theorem 2]{Yeadon}).
For an extension of Yeadon's result to  arbitrary (not necessarily semi-finite)  Neumann algebras, see \cite{Sherman}.
\begin{theorem}
\label{Yeadon}
\textbf{{\rm(\cite{Yeadon})}}
Let $1\leq p<\infty$ and $p\neq2$. 
A  mapping $U:L_p(\N)\to L_p(\N)$ is 
a linear surjective isometry if and only if there exists a unique pair $(u,\theta)$
consisting of  a unitary $u\in \N$ and 
an  automorphism or an anti-automorphism $\theta$ of $\N$ such that 
$$U(x)=u\theta(x) \tout x\in \N.$$
\end{theorem}

As this result is not explicitly stated  in  \cite{Yeadon}, we indicate how 
it follows from there. Since $U$ is surjective,  by  \cite[Theorem 2]{Yeadon}, there exist
 a normal Jordan isomorphism $J: \N\to \N$, a unitary $u\in\N$, and a positive self-adjoint operator $B$ affiliated with $\N$
 such that $U(x)=uBJ(x)$ for all $x\in\N.$ Since $\N$ is factor,  $J$  is either an automorphism  or an anti-automorphism (see \cite[Proposition 3.2.2]{BraRob}) and $B=1.$

 The group $\OO(L_p(\N))$ of linear bijective isometries of $L_p(\N)$  is a topological group 
 when equipped  with the strong  operator topology; this is the topology for which  
 a fundamental system family of neighbourhoods of $U$ in $\OO(L_p(\N)) $ is given
by  subsets of the form 
$$\{V\in \OO(L_p(\N)) \, : \, \Vert V(x_i)- U(x_i)\Vert_p \leq \eps \}$$ for $x_1,\dots, x_n\in L_p(\N)$
and $\eps>0$. 

We identify the extended automorphism group $\ExAut (\N)$,
the automorphism group ${\Aut}(\N)$ and 
the unitary group $\UU(\N)$ of $\N$ with  subgroups of $\OO(L_p(\N))$, 
 endowed with  the topology induced by that of   $\OO(L_p(\N)).$
As is easy to show, the topology on $\UU(\N)$ coincides with the topology induced
 by its embedding in $L_p(\N)$ given by $u\mapsto u1.$
 
 For $U\in \OO(L_p(\N))$, we will often write  $U= (u, \theta)$ for $u$ in  $\UU(\N)$
and $\theta$ in ${\ExAut}(\N)$ 
and refer to the pair $(u, \theta)$ as the \emph{Yeadon decomposition} of $U$.

The set of isometries $U$ with Yeadon decomposition $(u, \theta)$ for which $\theta$
is an  automorphism  of $\N$ is a closed subgroup of 
 index at most 2 in $\OO(L_p(\N))$ and will be denoted by $\OO^+(L_p(\N)).$
 
Observe that $\UU(\N)$ is normal in  $\OO^+(L_p(\N))$  but, in general, 
 not normal in $\OO(L_p(\N)).$

It follows from Yeadon's result and from  \cite[ Theorem~2]{Junge} that 
the subgroup $\OO^+(L_p(\N))$ can be intrinsically characterized 
inside $\OO(L_p(\N)$ as the subgroup of the complete
isometries (or as the  subgroup of $2$-isometries)
of $L_p(\N)$ in the sense of operator spaces,
 that is, the isometries $U$ of  $L_p(\N)$ such that 
${\rm id}\otimes U$ is an isometry of $L_p(M_n(\CCC)\otimes \N)$ for every 
$n\in \NN$ (or such that  ${\rm id}\otimes U$ is an isometry of $L_p(M_2(\CCC)\otimes \N)$)
 It should be mentioned that  completely isometric or, more generally, completely bounded mappings
are natural objects to study in the context of operator algebras (see \cite{Paulsen}).

\begin{corollary}
\label{Cor-Yeadon}
For $1\leq p<\infty$ and $p\neq2$, the group $\OO^+(L_p(\N))$ 
is isomorphic as topological  group to the topological semi-direct product $\UU(\N)\rtimes {\Aut}(\N),$ 
given by the natural action of  ${\Aut}(\N)$ on  $\UU(\N).$
\end{corollary}
\begin{proof} 
The fact that  $\OO^+(L_p(\N))$ 
is isomorphic as  abstract group to  $\UU(\N)\rtimes {\Aut}(\N)$ 
is a consequence of Yeadon's theorem. Moreover, evaluation at $1\in\N$ shows  that the projection
$$\OO^+(L_p(\N))\cong \UU(\N)\rtimes {\Aut}(\N)\to \UU(\N),\qquad  (u, \theta)\mapsto u$$
 is continuous. Hence, the projection  $\OO^+(L_p(\N))\to  {\Aut}(\N)$
 is also continuous; so, $\OO^+(L_p(\N))$ and $\UU(\N)\rtimes {\Aut}(\N)$
 are isomorphic as topological groups.
 $\bsq$
\end{proof} 


Every $U=(u, \theta)$ in $ \OO(L_p(\N))$ defines, for every $1\leq q <\infty,$
a  linear bijective isometry of $L_q(\N)$ by the same formula:
$U(x)=u\theta(x)$ for all $x$ in the dense subspace $\N$ of $ L_q(\N).$ 
One obtains in this way a mapping 
$$\Phi_{p,q}:  \OO(L_p(\N))\to  \OO(L_q(\N)).$$
For $q\neq 2,$ this mapping is of course surjective;
this is not the case for $q=2$ (if $\N$ is infinite dimensional)
and we define the $\N$-unitary groups  $\OO_{\N}(L_2(\N))$ and  $\OO^+_{\N}(L_2(\N))$  of the Hilbert  space $L_2(\N)$ 
to be the images of   $\OO(L_p(\N))$ and $\OO^+(L_p(\N))$
under $\Phi_{p,2}$;  thus,
$\OO_{\N}(L_2(\N))$  (respectively $\OO^+_{\N}(L_2(\N))$)  is the group of unitary operators $U$
of $L_2(\N)$  which have a Yeadon decomposition $U=(u, \theta)$
for   $u\in \UU(\N)$ and $\theta\in  {\ExAut}(\N)$ (respectively $\theta\in  {\Aut}(\N)$).

The next proposition will allow us to transfer representations in $\OO(L_p(\N))$ 
to representations in the $\N$-unitary group $\OO_{\N}(L_2(\N))$ of  $L_2(\N).$ Its proof
uses in a crucial way properties of the  \emph{Mazur map}, which  
is the (non linear) mapping $M_{p,q}: L_p(\N)\to L_q(\N)$  defined by  
$$M_{p,q}(x) = u|x|^{\frac{p}{q}}$$
 for 
$x\in L_p(\N)$ with polar decomposition $x=u|x|.$

\begin{proposition}
\label{Pro-Mazur}
Let $1\leq p, q<\infty$ and $p, q\neq2$. 
The mappings 
\begin{align*}
 &\Phi_{2,p}:  \OO_{\N}(L_2(\N))\to \OO(L_p(\N)),  \Phi_{p,2}:  \OO(L_p(\N))\to  \OO_{\N}(L_2(\N)),\\
& \text{and}\qquad\Phi_{p,q}:  \OO(L_p(\N))\to \OO(L_q(\N)) 
\end{align*}
are continuous.  In particular, the groups
$\OO(L_p(\N))$ for $p\neq 2$ and $\OO_{\N}(L_2(\N))$  are mutually isomorphic 
as topological groups. 

\end{proposition} 

\begin{proof}
It suffices to prove that the mappings  $\Phi_{2,p}, \Phi_{p,2}$ and $ \Phi_{p,q}$  
are continuous on the open subgroups  $\OO^+_{\N}(L_2(\N))$
and  $\OO^+(L_p(\N))$ for $p\neq 2$.

For  $U\in \OO^+(L_p(\N))$ with Yeadon decomposition $(a, \theta)$ and $x\in \N$  with polar decomposition $x=u|x|,$   
we have 
\begin{align*}
 U(x)&= a\theta(u)\theta(|x|) = a\theta(u) \theta(|x|^{\frac{q}{p}})^{\frac{p}{q}} \\
 &= M_{p,q}(a\theta (M_{q,p} (x)) = M_{p,q}\circ U\circ M_{q,p} (x),
\end{align*}
so that  
$$
\Phi_{p,q}(U)= M_{p,q}\circ U\circ M_{q,p} \qquad \text{for all}\  \  U\in \OO^+(L_p(\N)).
$$
It is known that (even for a general von Neumann algebra $\N$) 
the restriction  $M_{p,q}: B_1(L_p(\N)) \to B_1(L_q(\N)) $ of $M_{p,q}$ to the unit ball $B_1(L_p(\N))$ of $L_p(\N))$
 is uniformly continuous (see  \cite[Lemma 3.2]{Raynaud}; a more precise result is  
 proved in \cite{Ricard}: $M_{p, q}$ is $\min\{p/q, 1\}$--H\"older continuous on  $B_1(L_p(\N))$).
Since
\begin{align*}
\Vert \Phi_{p,q} (U)(x) -\Phi_{p,q} (V)(x)\Vert_q=  \Vert M_{p,q}(U(M_{q,p} (x))-M_{p,q}(V(M_{q,p} (x))\Vert_q,
\end{align*}
for  $U, V\in  \OO^+(L_p(\N))$  and $x\in L_q(\N)$, the  proposition follows. $\bsq$
\end{proof} 

\subsection{Group representations by linear complete isometries on $L_p(\N)$}
\label{SS:Rep}

Let $\Ga$ be a discrete group and $p\in [1,+\infty[,\, p\neq 2,$ fixed throughout this section.

For  a mapping $\pi: \Ga\to \OO(L_p(\N))$ or  $\pi: \Ga\to \OO_\N(L_2(\N))$, we have corresponding mappings $u: \Ga\to \UU(\N)$ and  $\theta: \Ga\to \ExAut(\N)$
given by  the Yeadon decomposition $\pi(g) = (u_g, \theta_g)$ for every $g\in \Ga.$
We will refer to $\pi=(u, \theta)$ as the Yeadon decomposition  of $\pi$.
Observe that, if $\pi$ is a homomorphism, then $\theta:\Ga\to \ExAut(\N)$ is in general not  a homomorphism;
however, if  $\pi$ takes its values in $\OO^+(L_p(\N))$,  then $\theta:\Ga\to \Aut(\N)$ is indeed a homomorphism.

Given a group homomorphism $\theta: \Ga\to {\Aut}(\N),g\mapsto \theta_g,$ 
we denote by $Z^1(\Ga, \theta)$ the set of all corresponding  $1$-cocycles, that is,
the set of mappings $u: \Ga\to \UU(\N)$ such that 
$$u_{gh}= u_{g}\theta_{g}(u_{h}) \tout g,h\in\Ga.$$
Two $1$-cocycles $u$ and $v$ are cohomologous, if there exists $w\in \UU(\N)$ such that 
$$ v_g= w u_g\theta_g(w^*)\tout g\in \Ga.$$
The set of $1$-coboundaries $B^1(\Ga, \theta)$ is the set of $1$-cocycles  which are cohomologous
to the trivial cocycle $g\mapsto 1.$


The proof of following proposition is straightforward. 
\begin{proposition}
\label{Pro-Mazur1} Let $\pi: \Ga\to \OO^+(L_p(\N))$ or $\pi: \Ga\to \OO^+_\N(L_2(\N))$ be a mapping with Yeadon decomposition $\pi=(u, \theta).$
The following conditions are equivalent.
\begin{itemize}
\item[(i)] $\pi$ is a group homomorphism;
\item[(ii)] $\theta: \Ga\to  \Aut(\N)$ is a group homomorphism and 
$u: \Ga\to \UU(\N)$ is a $1$-cocycle with respect to $\theta.\, \bsq$ 
\end{itemize}
\end{proposition}

Given $\pi\in \Hom (\Ga, \OO^+(L_p(\N)),$ with Yeadon decomposition $(u,\theta)$,
 there are two  associated actions $\pi^l$ and  $\pi^r$
in $\Hom (\Ga, \Aut(\N)$ defined by 
$$\pi^l(g)= \theta_g \ \  \text{and}\ \  \pi^r(g)= \Ad(u_g) \theta_g \tout g\in\Ga.$$
(For $u\in \UU(\N),$  $\Ad(u)$ denotes the automorphism of $\N$ given by 
$\Ad(u)x= uxu^*$ for $x\in \N.$)

  Recall that $G=\OO^+(L_p(\N))$ or $G=\OO^+_\N(L_2(\N))$ acts on the set of all mappings $\pi:\Ga\to   G$
 by conjugation ${\rm Ad}(U)\pi (g)= U\pi(g) U^{-1}$ for  $U\in G,\, g\in\Ga.$

\begin{proposition}
\label{Pro-Mazur2} 
Let $\pi$ and $\rho$ be homomorphisms 
from $\Ga$ to  $G=\OO^+(L_p(\N))$  or  $G=\OO^+_\N(L_2(\N))$, with Yeadon decompositions $\pi= (u, \theta)$ 
and $\rho = (v, \alpha).$ The following conditions are equivalent:
\begin{itemize}
\item[(i)]  $\rho$ belongs to  the  $G$-orbit  of $ \pi$;
\item[(ii)] there exists $\vfi\in \Aut(\N)$ such that   $\alpha=\Ad(\vfi)(\theta)$
and such that $v$ is cohomologous to $\Ad(\vfi)(u):g\mapsto \vfi(u_g)$ in $Z^1(\Ga, \Ad(\vfi)(\theta)).$
\end{itemize}
\end{proposition}

\begin{proof}
For an element $U=(w, \vfi)$ in $G,$
 one computes that the Yeadon decomposition $(v, \alpha)$ of $\Ad(U)\pi$ 
 is given by
 $$
 \alpha_g= \Ad(\vfi)(\theta_g)= \vfi \theta_g \vfi^{-1}, \quad v_g= w\vfi(u_g) (\Ad(\vfi)(\theta_g)) (w^*)
 $$
 for every $g\in \Ga$ and the claim follows.$\bsq$
\end{proof}

 \subsection{Groups and factors with Kazhdan's Property (T)}
\label{SS:Kazhdan}
 We recall (see \cite{BHV}) that a (discrete) group $\Ga$ has Kazhdan's Property (T)
 if there exist a  finite subset $S$ of $\Ga$ and $\eps >0$ with the following property: 
if a  unitary representation $\pi: \Ga \to \UU(\H)$ of $\Ga$ in a Hilbert space $\H$
has a  $(S,\eps)$-invariant unit vector, 
that is,  a unit vector $v\in \H$ with 
$$ \Vert \pi(s)v-v\Vert \leq \eps \tout s\in S,
$$
 then 
there exists a non-zero $\Ga$-invariant vector in $\H$.
The pair $(S, \eps)$ is called a \emph{Kazhdan pair} for $\Ga.$ 
Moreover, if this is  the case, then for every $\delta>0$ and every $(S,\delta\eps)$-invariant 
unit vector $v,$  there exists a $\Ga$-invariant vector $w\in \H$ with 
$\Vert v- w\Vert\leq\delta$ (see Proposition~1.1.9 in \cite{BHV}).

We shall need the extension from \cite{NPS} of the previous result to projective unitary representations;
recall that a projective unitary representation of $\Ga$ in a Hilbert space $\H$
is a mapping $\pi$ from  $\Ga$ to the unitary group $\UU(\H)$ of $\H$ such that,
for every $g,h\in G,$  there exists a scalar $\mu_{g,h}\in {\mathbf S}^1=\{\la\in \CCC : \ |z|=1\}$
 with 
 $$\pi(g)\pi(h)= \mu_{g,h} \pi(gh)\tout g,h\in \Ga.$$ 
 A projective unitary representation $\pi$
determines a homomorphism $\widetilde{\pi}:\Ga\to \PP\UU(\H) $
to the projective  unitary group $\PP\UU(\H)=\UU(\H)/\mathbf{S}^1$ of $\UU(\H),$ where
$\mathbf{S}^1$ is identified with  the subgroup  of scalar
multiples of the identity operator ${\rm Id}_\H;$ conversely, every lift $\pi:\Ga\to \UU(\H)$ of  a homomorphism $\widetilde{\pi}:\Ga\to \PP\UU(\H)$
 is a projective unitary representation of $\Ga.$
 The mapping $\mu: \Ga\times \Ga \to \mathbf{S}^1, (g,h)\mapsto \mu_{g,h}$
 is a $2$-cocycle, that is, it satisfies the identity 
 $$
 \mu_{h,k}\mu_{g, hk}= \mu_{g,h}\mu_{gh, k}\, \, \text{for all} \, g,h, k\in \Ga.
 $$
 If $\mu: \Ga\times \Ga \to \mathbf{S}^1$ is a $2$-coboundary, that is, if there exists
 a mapping $\la: \Ga\to \mathbf{S}^1$ such that 
 $$\mu_{g,h} =\la_g \la_h \overline{\la_{gh}} \tout g,h\in \Ga,$$
  then $\pi$ gives rise to a genuine representation 
 $\overline{\pi}: \Ga\to \UU(\H),$ defined by 
 $$
 \overline{\pi}(g) = \overline{\la_g} \pi(g) \,  \, \text{for all} \, g\in \Ga
 $$
and inducing the same homomorphism $\Ga\to \PP\UU(\H)$ as $\pi.$

 Given a projective unitary representation $\pi:\Ga\to \UU(\H)$ and a subset $S$ of $\Ga$ and $\eps>0,$
 we will say that a unit vector $v\in \H$ is \emph{projectively $(S,\eps)$-invariant} 
 if,  for every $s\in S,$
 there exists $\alpha_s\in \CCC$ such that $\Vert \pi(s) v-\alpha_sv\Vert \leq\eps.$
 
 The following result  is  proved in  \cite{NPS}
 in the more general situation of  a pair of groups with the relative Property (T).
 When $\Ga$ has Property (T), with a Kazhdan pair $(S, \eps),$
  one checks easily that the proof of  Lemma 1.1 of \cite{NPS} yields exactly 
 the following result.

 \begin{theorem}
 \label{Theo-PopaEtAl}
 \textbf{{\rm (\cite{NPS})}} Let $\Ga$ be a Kazhdan group, with a Kazhdan pair $(S, \eps).$ 
 Fix $\delta$ with $0<\delta <1.$ Let  $\pi: \Ga \to \UU(\H)$  be a projective unitary representation of $\pi$,
 with corresponding $2$-cocycle $\mu: \Ga\times \Ga \to \mathbf{S}^1$, 
 and let $v\in \H$ be unit vector which is projectively $(S,\eps\delta^2/56)$-invariant.
 Then there exists a mapping $\la: \Ga\to \mathbf{S}^1$ with $\mu_{g,h} =\la_g \la_h \overline\la_{gh}$
 for all $g,h\in \Ga$  and a vector $v_0\in \H$  such that  
 $$\Vert v-v_0\Vert \leq\delta \qquad\text{and}\qquad\pi(g) v_0= \la _g v_0$$
  for all $g\in \Ga.$
 In particular, $\mu$ is a coboundary.
\end{theorem}

We now recall Property (T) for  von Neumann algebras as defined in  \cite{CoJo}.

Let $\N$ be  finite factor.  A \emph{Hilbert bimodule}
over $\N$ is a Hilbert space $\H$ carrying two commuting 
normal representations, one  of  $\N$ and one of the opposite algebra $\N^0;$ we will 
write 
$$v\mapsto xvy   \qquad \text{for all} \  v\in \H,\  x,y\in \N.$$ 
The factor $\N$ is said to have   Property (T) if there
exist a finite subset $F$ of $\N$ and $\eps'>0$
such that the following property  holds:
if a Hilbert bimodule $\H$ for  $\N$ contains a unit vector 
$v$ which is $(F,\eps')$-central, that is, which is such that 
$$\Vert xv-vx\Vert\leq \eps' \tout x\in F,$$
then $\H$ has a non-zero central vector,
that is, a non-zero vector $w\in \H$
such that $xw =w x$ for all $x\in \N.$
Moreover, one can choose $(F,\eps')$ such that
for every $\delta>0$ and every $(F,\delta\eps')$-central
unit vector $v,$   there exists a central  vector $w\in \H$ with 
$\Vert v- w\Vert\leq\delta$ (see Proposition~1 in \cite{CoJo}).
We call such a pair $(F, \eps')$  a \emph{Kazhdan pair} for $\N.$

It was   shown in  \cite{CoJo} (see Proposition~12.1.19 in \cite{Brown-Ozawa}) that the subgroup $\Inn({\N})$  of inner automorphisms
of $\N$ (that is, the  subgroup of  automorphisms of the form
$\Ad(u)$ for $u\in \UU(\N)$) is open  in $\Aut (\N).$
  Here, $\Aut (\N)$ is endowed with the topology of pointwise 
$L_2$-convergence (we could also take the induced topology from $\OO^+(L_p(\N))$ as
above for any $1\leq p<\infty$). 
We will need a quantitative estimate, in terms of a Kazhdan pair $(F, \eps'),$
for  the distance to $1$ of the  unitary operators
defining the appropriate inner automorphisms.

\begin{proposition}
\label{Pro-Kazhdan1} 
Let $\N$ be a finite factor with Property (T). 
Let $(F, \eps')$ be a Kazhdan pair for $\N$ .
Let  $0<\delta <1$  and  let  $\V_{\delta}$ be the neighbourhood  of the trivial automorphism ${\rm id}_{\N}$  
given by
$$\V_{\delta} = \{\theta\in \Aut(\N)\, : \, \Vert \theta(x)-x \Vert_2\leq \eps'\delta/2  \ \text{for all}\ x \in F\}.$$
Then $\V_{\delta}$ is contained in $\Inn({\N}).$  More precisely, for every $\theta$ in $\V_{\delta},$ there exists $u$ in $ \UU(\N)$ with $\theta = \Ad(u)$ and
 $\Vert u- 1\Vert_2 \leq \delta.$

\end{proposition}
 
 \begin{proof}
 We follow the  standard proof that $\Inn (\N)$ is open
in $\Aut (\N)$  as given, for instance, 
in  the proof of  Proposition 12.1.19 in \cite{Brown-Ozawa}.

Let $\theta \in \V_{\delta}.$ We define a bimodule structure on  $L_2(\N)$  over $\N$ by
 $$v\mapsto \theta(x)vy  \  \ \text{for all} \ \  v\in L_2(\N), \   x,y\in \N.$$ 
 Then $1\in L_2(\N)$ is a unit vector which is $(F,\eps'\delta/2)$-central.
 Hence, there exists a central vector $w\in L_2(\N)$ with 
$\Vert w- 1\Vert_2\leq\delta/2.$  
Let $w=u|w|$ be the polar decomposition of $w,$ viewed as a densely defined operator on $L_2(\N)$  affiliated
to $\N.$  Then, $|w|$ is in the center of $\N$ and hence $|w|=\la 1$ for some $\la>0.$
It follows that $u$ is a unitary element in $\N$  such that $\theta = \Ad(u).$
As $ \Vert w\Vert_2= \la,$ we have 
$ |1-\la|\leq \Vert w- 1\Vert_2 \leq \delta/2$ and therefore
$$
\Vert u- 1\Vert_2  \leq \Vert u - w\Vert_2+ \Vert w- 1\Vert_2=|1-\la| +\Vert w- 1\Vert_2  \leq  \delta.\ \bsq
$$

 \end{proof}

\subsection{Projective $1$-cocycles for actions of Kazhdan groups}
\label{SS:ProjCocy}

 In the sequel, we will need to deal with
 mappings $ \Ga\to \UU(\N)$ 
 which are $1$- cocycles for an action  of  $\Ga$ on $\N$ modulo scalars
 in the following sense (cocycles of this type appear in Section 1.3 of \cite{Popa}, where they are called weak $1$-cocycles).

 \begin{definition}
 \label{Def-ProjCocycle}
 Let $\Ga$ be a group, $\N$  a von Neumann algebra, and $\theta: \Ga\to \Aut(\N)$ a homomorphism.
 A \emph{projective $1$-cocycle}  for $\theta$ is a mapping $u: \Ga\to \UU(\N)$ such 
 that, for every $g,h\in \Ga,$  there exists a scalar $\mu_{g,h}\in {\mathbf S}^1$
 with 
 $$u_g \theta_g(u_h)=  \mu_{g,h} u_{gh}.$$
 Two projective $1$-cocycles $u$ and $v$  are cohomologous if there exist 
 $w$ in  $\UU(\N)$ and a mapping $\la: \Ga\to {\mathbf S}^1$ such that 
$$ v_g= \la_gw u_g\theta_g(w^*) \tout g\in \Ga.$$
 A \emph{projective coboundary} is a projective $1$-cocycle which is cohomologous to the
 trivial cocycle $g\mapsto 1.$
 \end{definition}

 The following lemma, which can be checked by a straightforward computation, shows that 
 projective cocycles appear naturally. 
 \begin{lemma}
 \label{Lem-ProjCocy}
 Let $\Ga$ be a group, $\N$  a factor, and $\theta: \Ga\to \Aut(\N)$ a homomorphism.
 For a mapping  $u:\Ga\to \UU(\N),$ the following properties are equivalent.
 \begin{itemize}
 \item[(i)] $u$ is a projective $1$-cocycle for $\theta;$
 \item[(ii)] the mapping $g\mapsto \Ad(u_g)\theta_g$ is a homomorphism from $\Ga$ to  $\Aut(\N).\bsq$
 \end{itemize}
 \end{lemma}

We denote by $Z^1_{\rm proj}(\Ga, \theta)$
and by  $B^1_{\rm proj}(\Ga, \theta)$ the set of projective $1$-cocycles 
and coboundaries for $\theta$.
We equip $Z^1_{\rm proj}(\Ga, \theta)$  with the topology 
of pointwise $L_2$-convergence:  a sequence $(u^{(n)})_n$ in $Z^1_{\rm proj}(\Ga, \theta)$ converges to 
$u\in Z^1_{\rm proj}(\Ga, \theta)$ if $\lim_n \Vert u^{(n)}_g -u_g\Vert_2=0$  for every $g\in \Ga.$

 Assume now that $\Ga$ has Property (T). We will need to know that  cohomology classes  in $Z^1_{\rm proj}(\Ga, \theta)$ 
are  open.  This is not true in general even
for classes in $Z^1(\Ga, \theta)$ and even when $\theta$ is ergodic (see Examples 4 and 8 in \cite{Kawa}).   
However, the following result was shown  in \cite[Theorem~7]{Kawa}.
Let $u\in Z^1(\Ga, \theta)$ be such that the action of $\Ga$ on $\N$ given by 
$g\mapsto \Ad(u_g) \theta_g$ is ergodic; then the equivalence class  of   $u$ 
 is open in  $Z^1(\Ga, \theta).$ 
 Following the same proof and making crucial use of Theorem~\ref{Theo-PopaEtAl},
 we now show that a quantitative version of this result is  true for  projective $1$-cocycles.

\begin{proposition}
\label{Pro-Kawa}
Let $\Ga$ be a group with Kazhdan's Property (T),  with a Kazhdan pair $(S, \eps).$ 
 Let $\N$ be a finite factor and $\theta: \Ga\to \Aut(\N)$ a homomorphism.
Let $u:\Ga\to \UU(\N)$ be a projective $1$-cocycle  for $\theta$. 
Assume that the action of $\Ga$ on $\N$ given by  $g\mapsto \Ad(u_g)\theta_g$
is ergodic. Fix   $0<\delta<1$ and let  $\U_{\delta}$  be the  neighbourhood  of $u$  in $Z^1_{\rm proj}(\Ga, \theta)$
defined by 
$$\U_{\delta} = \{ v\in Z^1_{\rm proj}(\Ga, \theta) : \,  \Vert  v_s- u_s \Vert_2\leq \eps\delta^2/224  \ \text{for all}\ s \in S\}.$$
Then every $v\in \U_{\delta}$ is cohomologous to $u.$ More precisely, for  every $v\in \U_{\delta},$ there exists $w\in \UU(\N)$ 
with  
$\Vert w- 1\Vert_2 \leq \delta$  and a mapping $\la: \Ga\to {\mathbf S}^1$ such that 
$ v_g= \la_gw u_g\theta_g(w^*)$ for all $g\in \Ga.$ 
\end{proposition}

 \begin{proof}
We  adapt the proof from \cite[Theorem~7]{Kawa}, making it quantitative at the appropriate places. 
Let $v \in \U_{\delta}.$ For every $g\in\Ga,$
let $\pi(g)$ be the unitary operator  on $L_2(\N)$ given by 
$$
\pi(g) x= u_g \theta_g(x) v_g^* \tout x\in \N.
$$
Since $u$ and $v$ are projective $1$-cocycles for $\theta,$ 
the mapping $\pi: g\mapsto \pi(g)$ is a projective unitary representation of $\Ga,$ as is easily checked.
Let $w: \Ga\times \Ga\to   \mathbf{S}^1$ be the corresponding $2$-cocycle.
Observe that $1\in L_2(\N)$ is a unit vector which is $(S,\eps\delta^2/224)$-invariant.
Hence, it follows from Theorem~\ref{Theo-PopaEtAl} that there exists
a mapping $\la: \Ga\to \mathbf{S}^1$ with 
$$\mu_{g,h} =\la_g \la_h \overline\la_{gh} \tout g,h\in \Ga$$
  and a vector $b\in L_2(\N) $  such that  
 $\Vert b-1\Vert \leq\delta/2$ and  $\pi(g) b= \la _g b$ for all $g\in \Ga.$
Thus, $b\neq 0$ and $u_g\theta_g(b) v_g^*=\la_gb$ for every $g\in\Ga.$  
We view $b$ as a densely defined operator on $L_2(\N)$  affiliated
to $\N.$  Taking adjoints, we see that the positive operator $bb^*\in L^1(\N)$ is fixed by
the extension to $L^1(\N)$ of $\Ad(u_g) \theta_g$  for every
$g\in \Ga$. Since $g \mapsto \Ad(u_g)\theta_g $ is ergodic, it
follows that $bb^*= \beta 1$ for some $ \beta>0$
Then $w:= b^*/\sqrt{\beta}$ is a unitary element in $\N$ such that 
$$v_g= \overline{\la_g}wu_g\theta_g(w^*) \tout g\in\Ga.$$
Moreover, as in the proof of Proposition~\ref{Pro-Kazhdan1}, we have 
$\Vert w- 1\Vert_2\leq  \delta.\, \bsq$
 \end{proof}

One can improve upon  the constant defining $\U_\delta$  in the previous proposition, when 
one deals with  genuine $1$-cocycles instead of projective ones; 
indeed, in this case, the projective unitary representation 
appearing in the proof is a true  unitary representation
and one checks that the following statement holds.

\begin{proposition}
\label{Prop-Kawa2}
Let $\Ga$,  $(S, \eps)$,  $\N$ and $\theta$ be as 
in Proposition~\ref{Pro-Kawa}.  
Let $u:\Ga\to \UU(\N)$ be a  $1$-cocycle  for $\theta$
such that   $g\mapsto \Ad(u_g)\theta_g$
is ergodic. For  $0<\delta<1,$ set 
$$\U_{\delta} = \{ v\in Z^1(\Ga, \theta) : \,  \Vert  v_s- u_s \Vert_2\leq \eps\delta/2  \ \text{for all}\ s \in S\}.$$
Then, for  every $v\in \U_{\delta},$ there exists $w\in \UU(\N)$ 
with   $\Vert w- 1\Vert_2 \leq \delta$  such that 
$$ v_g= w u_g\theta_g(w^*)\tout g\in \Ga. \, \bsq$$
\end{proposition}

 \section{Proof of Theorem~\ref{Theo1}}
 \label{S:ProofTheo1}
Let  $\pi: \Ga\to \OO(L_p(\N))$ be a   group homomorphism for $p\neq 2.$
Then $\Phi_{p,2}\circ \pi$ is a group homomorphism from $\Ga$ to the
$\N$-unitary group $\OO_{\N}(L_2(\N))$  as defined in Section~\ref{SS:Isometries}, where
$\Phi_{p,2}:  \OO(L_p(\N))\to  \OO_{\N}(L_2(\N))$ is the identity mapping.
 By Proposition~\ref{Pro-Mazur}, $\OO(L_p(\N))$
and $\OO_{\N}(L_2(\N))$ are topologically isomorphic groups.
Hence,  
to prove that $\pi: \Ga\to \OO(L_p(\N))$ is locally rigid amounts to prove
that $\Phi_{p,2}\circ\pi: \Ga\to \OO_{\N}(L_2(\N))$
is locally rigid. So, we can  replace $\pi$ by $\Phi_{p,2}\circ \pi.$  

Set $\Ga^+:= \pi^{-1}(\OO^+_{\N}(L_2(\N));$ 
then $\Ga^+$ is a normal  subgroup 
of index at most 2 in $\Ga.$

 Let $\pi=(u, \theta)$  be the Yeadon decomposition
of $\pi.$ 
  Recall that the associated  homomorphisms $\pi^l, \pi^r\in\Hom (\Ga^+, \Aut(\N))$ are given by
$$\pi^l(g)= \theta_g \qquad \text{and}\qquad \pi^r(g)= \Ad(u_g)\theta_g$$
for every $g\in\Ga^+.$

Assume now that $\Ga$  and $\N$ have both Property (T) and that 
 $\pi^l$ and  $\pi^r$  are ergodic. 

We first prove  Theorem~\ref{Theo1}  in  the case where $\pi$ takes its values in $\OO^+_{\N}(L_2(\N))$
and will then reduce the general case to this situation.

\subsection{{The case $\Ga= \Ga^+$}} 
In this case,  $\theta\in \Hom (\Ga, \Aut(\N))$ and $u\in Z^1(\Ga,\theta);$ see Section~\ref{SS:Rep}.
Recall that, by  Corollary~\ref{Cor-Yeadon} and Proposition~\ref{Pro-Mazur},
 $\OO^+_\N(L_2(\N))$ is topologically isomorphic to the topological semi-direct product $\UU(\N)\rtimes {\Aut}(\N).$

Let $(S, \eps)$ be  a Kazhdan pair for $\Ga$  and 
$(F, \eps')$ a Kazhdan pair for  $\N$. We can assume that $S$ is a  generating set for $\Ga,$
since $\Ga$ is finitely generated (see Theorem~1.3.1 in \cite{BHV}).
Moreover, since $\Aut(\N)$ is open in $\ExAut(\N),$ upon enlarging  $F$ 
and  reducing $\eps'$ if necessary, we  can also assume
that $$\max_{x\in F}  \Vert\tau(x)- x\Vert_2 >\eps' $$
 for every anti-automorphism $\tau$ of $\N.$

 Fix $0<\delta <1$ and define $\V_\delta $  to be the neighbourhood  of $\pi$ in $\Hom (\Ga, \OO_\N(L_2(\N)))$
 consisting of all $\rho\in  \Hom (\Ga, \OO_\N(L_2(\N)))$ with Yeadon decomposition $\rho=(v, \alpha)$ such that 
 \begin{align*}
 &(1)\ \ \Vert v_s-u_s\Vert_2 \leq \delta\eps/4 \ \ \text{for all}\  s\in S\ \text{and}\\
 &(2)\ \ \Vert\alpha_s(x)- \theta_s(x)\Vert_2\leq \, \frac{\delta^2\eps^3\eps'}{28672}=\frac{\delta^2\eps^3\eps'}{2^7\cdot 224}  \ \ \text{for all}\  s\in S\ \text{and all} \ x\in F.
 \end{align*}

Let $\rho\in \V_\delta$ with Yeadon decomposition $\rho=(v, \alpha).$ 
Since, by (2),
$$ \max_{x\in F} \Vert(\theta_s^{-1}\alpha_s)(x)- x \Vert_2= \Vert\alpha_s(x)- \theta_s(x)\Vert_2\leq \eps',$$
it follows that $\alpha_s\in \Aut(\N)$ for all $s\in S$ and hence $\alpha$ takes its values in  $\OO^+_\N(L_2(\N)).$
Hence, $\alpha\in \Hom (\Ga, \Aut(\N))$ and $v\in Z^1(\Ga,\alpha).$

\bigskip
\noindent
\textbf{Claim.} There exist $a, b\in \UU(\N)$ with 
$$\Vert a-1\Vert_2\leq \delta \, \, \text{and} \, \,  \Vert b-1\Vert_2\leq \delta$$
 such that 
$$
\alpha_g=\Ad(a)\theta_g\Ad(a^*) \,\, \, \text{and} \, \,\, v_g =b au_ga^* (\Ad(a) \theta_g\Ad(a^*))(b^*)
$$
for all $g\in \Ga.$
In particular,  once proved, this claim  will show that $\rho$ is in the $\OO^+_\N(L_2(\N))$-orbit of 
$\pi$ (see Proposition~\ref{Pro-Mazur2}). 
The proof will be carried out in four steps.

\bigskip
\noindent
$\bullet$ {\it First  step:} We claim that
there exists a projective $1$-cocycle $w$ in $Z^1_{\rm proj}(\Ga, \theta)$ with the following properties:
\begin{align*}
&\alpha_g=\Ad(w_g)\theta_g \quad\text{for all}\quad g\in \Ga \quad\text{and} \\ 
 &\Vert w_s-1\Vert_2 \leq \frac{\delta^2\eps^3}{2^6\cdot 224}  \quad\text{for all}\quad s\in S.
 \end{align*}
Indeed, by (2) above,  for every $s\in S$ and $x\in F,$ we have
$$
\Vert\theta_s^{-1}(\alpha_s(x))- x\Vert_2= \Vert\alpha_s(x)- \theta_s(x)\Vert_2\leq \frac{\eps'\delta^2\eps^3}{2^7\cdot 224}.
$$ 
Hence, it follows from Proposition~\ref{Pro-Kazhdan1} that, for every $s\in S,$ there exists $w_s$ in $\UU(\N)$ 
with 
$$\Vert w_s-1\Vert_2 \leq \frac{\delta^2\eps^3}{2^6\cdot 224}$$
and such that $\alpha_s=\Ad(w_s)\theta_s.$ 

Now,  $\alpha$ and $\theta$ are group homomorphisms from $\Ga$ to $\Aut(\N)$ 
and  $\Inn(\N)$ is a normal subgroup  in $\Aut(\N).$ Moreover,  we have just shown that
the homomorphisms $p\circ \alpha$ and $p\circ \theta$ from $\Ga$ to 
the quotient group $\Aut(\N)/\Inn(\N)$ agree on  the generating set $S$, where 
$$p: \Aut(\N)\to \Aut(\N)/\Inn(\N)$$
is the canonical projection. It follows that  $p\circ \alpha= p\circ \theta$ on $\Ga.$
Hence, we can extend $S\mapsto \UU(\N), s\mapsto w_s$ to a mapping 
$w:\Ga\mapsto \UU(\N)$ such that $\alpha_g=\Ad(w_g)\theta_g$ for all $g\in \Ga.$
By Lemma~\ref{Lem-ProjCocy},  $w$  is a projective $1$-cocycle for $\theta.$ 

\medskip
\noindent
$\bullet$ {\it Second step:}  We claim that   there exist $a\in \UU(\N)$ and a mapping $\la: \Ga\to {\mathbf S}^1$ such that 
 such that 
$$w_g= \la_g a \theta_g(a^*) \tout g\in \Ga$$ 
(that is, $w$ is in  $B^1_{\rm proj}(\Ga, \theta)$) and such that 
$$\Vert a-1\Vert_2 \leq \frac{\delta\eps}{2^3}.$$
  Indeed, the action of $\Ga$
given by $\pi^l=\theta$ is ergodic and 
$$
\Vert w_s-1\Vert_2 \leq  \frac{\delta^2\eps^3}{2^6\cdot  224}=\left(\frac{\delta\eps }{2^3}\right)^2\frac{\eps}{224} \tout s\in S.
$$ 
for every $s\in S.$ Hence, the claim  follows from  Proposition~\ref{Pro-Kawa} applied to 
the trivial cocycle $u: g\mapsto 1.$

\medskip
\noindent
$\bullet$ {\it Third step:}  Let $a\in \UU(\N)$ be as in the second step. We claim that 
$\alpha=\theta^{\Ad(a)},$ that is, 
$$\alpha_g=\Ad(a)\theta_g \Ad(a^*) \tout g\in \Ga.$$
 Indeed, this follows from the fact that $\alpha_g=\Ad(w_g)\theta_g$  and 
$w_g= \la_g a\theta_g(a^*)$ for every $g\in \Ga$.

\bigskip
Let $u'=\Ad(a)u$ be the cocycle in $Z^1( \Ga, \alpha) = Z^1( \Ga, \theta^{\Ad(a)}) $
defined by 
$$u'_g= a u_g a^* \tout g\in\Ga.$$

\medskip
\noindent
$\bullet$ {\it Fourth step:}   We claim that there exists $b\in \UU(\N)$ 
with   $\Vert b- 1\Vert_2 \leq \delta$  such
$$ v_g= b u'_g\theta_g^{\Ad(a)}(b^*)\tout g\in \Ga.$$
Indeed,  by (1) above and the choice of $a$, we have
\begin{align*}
\Vert v_s-u'_s\Vert_2 &=\Vert v_s- a u_s a^* \Vert_2\leq  \Vert v_s- u_s\Vert_2 + \Vert au_s- u_sa\Vert_2 \\
& \leq \Vert v_s- u_s\Vert_2 + \Vert (a-1)u_s\Vert_2 + \Vert u_s(a-1)\Vert_2 \\
&\leq   \Vert v_s- u_s\Vert_2+ 2\Vert a-1\Vert_2 \\
&\leq \frac{\delta\eps}{4}+ 2 \frac{\delta\eps}{8}=\frac{\delta\eps}{2},
\end{align*}
for every $s\in S.$ 
Moreover, the action of $\Ga$
given  by $\pi^r(g)= \Ad(u_g) \theta_g$ for $g\in\Ga$  is ergodic.
Hence, the action 
given by $g\mapsto \Ad(u'_g)\alpha_g=\Ad(a) \pi^r(g) \Ad(a^*)$  is also ergodic.
The   claim  follows now from  Proposition~\ref{Prop-Kawa2}.

\subsection{{The case $\Ga\neq \Ga^+$}}

We assume now that $\Ga^+= \pi^{-1}(\OO^+_{\N}(L_2(\N))$ is a proper subgroup 
and hence a normal subgroup of index 2 in $\Ga.$ 
Observe that  $\Ga^+$ has also Property (T). Let $(S, \eps)$ be  a Kazhdan pair for $\Ga^+$.

 Since $\Inn(\N)$ is open in $\Aut(\N),$ the set  
 $$
\U =\{\varphi=(u, \Ad(v)) \in \OO^+_{\N}(L_2(\N)) \ : \ \Vert u -1\Vert_2<\sqrt{3}/4\quad\text{and}\quad \Vert v-1\Vert_2<\sqrt{3}/4\}
$$
is an open neighbourhood
of the identity in $\OO^+_{\N}(L_2(\N)).$

Fix $s_0\in \Ga\setminus \Ga^+.$
Let $0<\delta<\sqrt{3}/20.$ Define $\V=\V_{\delta}$  to be the neighbourhood  of $\pi$ in $\Hom (\Ga, \OO_\N(L_2(\N)))$
consisting of all $\rho=(v, \alpha)$ in  $\Hom (\Ga, \OO_\N(L_2(\N)))$ such that the conditions (1) and (2) from above hold
 and such that, moreover, 
 \begin{align*}
 (3)\qquad \rho({s_0}) \in \U \pi({s_0}).
\end{align*}

Let $\rho\in \V.$
We can apply the conclusion of the first case to the restrictions $\pi|_{\Ga^+}$ and $\rho|_{\Ga^+}$ of $\pi$ and $\rho$
to  $\Ga^+$ and conclude that there exists $U= (a, \Ad(b))$ in $\OO^+_{\N}(L_2(\N))$ with 
unitaries $a$ and $b$ in $\UU(\N)$ which are
 $\delta$-close to $1$ in the $L^2$-norm and such that $\pi|_{\Ga^+}= \Ad(U)(\rho|_{\Ga^+}).$

We claim that $\pi= \Ad(U)\rho.$ 
Indeed, set $\beta:= \Ad(U)\rho.$  By (3), there exists $\varphi_1=(a_1, \Ad(b_1))\in\U$ such that  
$\rho(s_0)= \varphi_1\pi(s_0).$ Hence, 
$$\beta(s_0)= U \varphi_1\pi(s_0)U^{-1}=U\varphi_1(\pi (s_0)U^{-1}\pi (s_0)^{-1}) \pi (s_0).$$
Set $\varphi_2: =\pi (s_0)U^{-1}\pi (s_0)^{-1}.$  One checks that
$\varphi_2 =(a_2, \Ad(b_2))$ for unitaries  $a_2$ and $b_2$ which are $4\delta$-close to $1,$
since $a$ and $b$ are $\delta$-close to $1$  in the $L^2$-norm.
Set $\varphi :=U\varphi_1 \varphi_2,$ so that $\beta(s_0)= \varphi\pi(s_0).$ 
Then $\varphi=(c, \Ad(d))$ 
for unitaries $c$ and $d.$ 
Since  $a_1$ and $b_1$ are $\sqrt{3}/4$-close to 1, 
since $a_2$ and $b_2$  are $4\delta$-close to $1,$
 and since $\delta <\sqrt{3}/20$, one checks that $c$ and $d$
  are $\sqrt{3}/2$-close to $1$ in the $L^2$-norm.
 
Using the fact that $\beta$ and $\pi$ are homomorphisms on $\Ga$
and coincide  on the normal subgroup $\Ga^+,$ 
 we have, for every $g\in \Ga^+,$
\begin{align*}
 \pi(s_0 g s_0^{-1})&=\beta(s_0 g s_0^{-1}) =\beta(s_0)\beta(g) \beta(s_0^{-1})= \varphi \pi(s_0 g s_0^{-1}) \varphi^{-1}.
\end{align*}
So, $\varphi$ commutes with $\pi(g)$ for all $g\in \Ga^+.$

The condition that $\varphi=(c, \Ad(d))$ commutes with $\pi(g)=(u_g, \theta_g)$   
means that 
$$
 cd u_g \theta_g(x) d^*= u_g \theta_g (cdxd^*) \, \, \, \text{for all} \, \, x\in \N. \, \, \, (*)
$$
Taking adjoints, we deduce that
$$
d\theta_g(x^* x) d^*=  \theta_g(dx^* xd^*), 
$$
that is,  $\theta_g (d^*)d$ commutes with $\theta_g(x x^*)$ for every 
$x\in \N.$ Since $\N$ is a factor, it follows that, for every $g\in \Ga^+,$ we have 
$\theta_g (d^*) d= \la_g 1$ 
for some scalar $\la_g$  with $|\la_g|=1.$
Using the fact that $g\mapsto \theta_g $ is a group homomorphism, we see that
$g\mapsto \la_g$ is a unitary character of $\Ga^+.$ 

Since $d$ is $\sqrt{3}/2$-close to 1, we have
$ |\la_g-1| < \sqrt{3}$ for all $g\in \Ga^+.$ 
As is well-known, this implies that $\la_g=1$ for all $g\in \Ga^+$ (indeed, the only subgroup
$G$ of the unit circle with $|z-1| < \sqrt{3}$ for all $z\in G$ is the trivial subgroup).

So, $d^*$ is fixed by the  automorphisms  $\theta_g $ for $g\in \Ga^+$ and hence
$d= \la 1$ for some scalar $\la$  with $|\la|=1,$ by ergodicity of $\pi^l.$ 
Hence, $\Ad(d)$ is the identity and we can assume that $d=1.$

From $(*)$, we then obtain that $c$ is fixed by the  automorphisms  $\pi^r(g)= u_g\theta_g u_g^*$
for $g\in \Ga^+$  and so $c=  \la 1$ for some scalar $\la$  with $|\la|=1,$ by ergodicity of $\pi^r.$ 
Hence, $\beta(s_0)= \la \pi(s_0).$ 
Since $s_0^2\in \Ga^+$ and therefore $\beta(s_0^2) = \pi(s_0^2),$ we see that   $\la^2= 1.$ 
As $\Vert c-1\Vert_2 <2,$
it follows that $\la=1,$  that is,  $\beta(s_0)=\pi(s_0).$
Hence, $\beta= \pi$ and
 the  proof of the theorem is complete.
$\bsq$

\begin{remark}
\label{Rem-Quantitive}
Let $\pi: \Ga\to \OO^+_{\N}(L_2(\N))$ be a homomorphism.
For a fixed $0<\delta <1,$ the  set $\V_\delta$ given above is a neighbourhood of $\pi$ in $\Hom(\Ga, \OO^+_{\N}(L_2(\N)))$
 such that every $\rho\in \V_\delta$ is conjugate to $\pi$ by some $U=(b, \Ad(a))$ in $\OO^+_{\N}(L_2(\N))$
 for which $a$ and $b$ are  $\delta$-close to the identity in the $L_2$-norm.
 
Using estimates for the Mazur map $M_{p, q}$,  we 
 can determine a neighbourhood $\V_{\delta, p}$ of $\pi$ in $\Hom (\Ga, \OO^+(L_p(\N)))$ 
 for $p\neq 2,$ with the same properties.
 Indeed, by  \cite{Ricard},  there exist constants $C=C_p$ (of order $p$)  and $D$ 
 (independent of $p$) such that 
 \begin{align*}
 \Vert M_{p,2}(x)-   M_{p,2}(y)\Vert_2  &\leq C\Vert x- y\Vert_p^{\alpha}\\
  \Vert M_{2,y} (x')-   M_{2,p}(y')\Vert_p  &\leq D \Vert x'- y'\Vert_2^{\beta},\\
   \end{align*}
where $\alpha=  \min\{ p/2, 1\}$ and $\beta= \min\{ 2/p, 1\},$
 for $x, y$ in the unit ball in $L_p(\N)$ and  $x', y'$ in the unit ball in $L_2(\N).$

 We can clearly assume that, for every $x$ in the Kazhdan set $F$ for $\N,$
 we have  $\Vert x\Vert_2=1$ and hence $\Vert M_{2,p}(\theta(x))\Vert_p=1$
 for every $\theta \in \Aut(\N),$  since  $\Vert M_{2,p}(\theta(x))\Vert_p =
 \Vert \theta (M_{2,p}(x))\Vert_p=\Vert x\Vert_2^{2/p}.$

 Set    $\delta_p= \left(\frac{\delta}{D}\right)^{1/\beta}$ and
 $$ \eps_p=\left(\frac{ \delta_p\eps}{4 C}\right)^{1/\alpha}, \, \, \, \,
 \eps_p'= \left(\frac{\delta_p^2\eps^3\eps'}{C\cdot 2^7\cdot 224}\right)^{1/\alpha}.
 $$
Let  $\V_{\delta_p}$ to be the neighbourhood  of $\pi=(u,\theta)$ in $\Hom (\Ga,\OO^+(L_p(\N))$
 consisting of all $\rho=(v, \alpha)$ in  $\Hom (\Ga,\OO^+(L_p(\N))$ such that 
 \begin{align*}
 &\ \ \Vert v_s-u_s\Vert_p \leq \eps_p \ \ \text{for all}\  s\in S\ \text{and}\\
 &\ \ \Vert\alpha_s(M_{2,p}(x))- \theta_s(M_{2,p}(x))\Vert_p\leq \,\eps'_p   \ \ \text{for all}\  s\in S\ \text{and all} \ x\in F.
 \end{align*}
Then  every $\rho\in \V_{\delta_p}$ is conjugate to $\pi$ by some $U=(b, \Ad(a))$ in $\OO^+(L_p(\N))$
 for which $a$ and $b$ are  $\delta$-close to the identity in the $L_p$-norm.
\end{remark}

\section{On the assumptions in the statement of Theorem~\ref{Theo1}}
\label{S:Comments}

We present counterexamples in relation with  the various assumptions
made in Theorem~\ref{Theo1}.
\begin{example} 
\label{Exa1}
If the finite factor $\N$ does not have Property (T), the conclusion of Theorem~\ref{Theo1}
may not be  true. Indeed, let $\R$ be the hyperfinite type $II_1$-factor.  
M. Choda constructed in \cite{Choda} a continuous  family $(\theta_t)_{t\in [0,1]}$ of  actions of 
the  group $\Ga=SL_n(\ZZ)$ for $n\geq 2$ (recall that $SL_n(\ZZ)$ has
Property (T) for $n\geq 3$) by automorphisms on  $\R$,
which are ergodic and mutually non conjugate in ${\rm Aut} (\R)$ for irrational $t.$ 
It follows from Proposition~\ref{Pro-Mazur2}  that,
 for any $1\leq p <\infty, p\neq 2$ and any irrational $t$, the homomorphisms  $\pi_t: \Ga \to \OO^+(L_p(\R))$ defined by these actions are   mutually non conjugate in $\OO^+(L_p(\R))$ and hence are not locally rigid.
In fact, a more general result   in \cite[Corollary 0.2]{Popa} implies that  any Kazhdan group admits  a continuous family of actions by automorphisms on  $\R$ which are mutually non conjugate.
\end{example} 

\begin{example} 
\label{Exa2}
The conclusion of Theorem~\ref{Theo1} might also fail, if  
 any one of the associated actions $\pi^l$ and  $\pi^r$ by automorphisms  of $\N$
  is not ergodic.  
 A counter-example may be obtained by a slight modification
 of Example 8 in \cite{Kawa} as follows. 
 
 Let $H$ be an ICC group   with Kazhdan's property (for instance $H= SL_3(\ZZ)$).
  Set $\Ga= H\times H$ and $\N= \N(\Ga).$ Then $\Ga$ is an ICC group with  Kazhdan's property  and $\N$
 can be identified with  the  tensor product  $\N(H)\overline{\otimes} \N(H)$ of von Neumann algebras,
 with trace $\tau= \tau_H\otimes \tau_H,$ where $\tau_H$ is the canonical trace on $\N(H).$
 For $1\leq p <\infty, p\neq 2,$ let $\pi_0$ denote the embedding  $ \Ga \to \OO^+(L_p(\N))$ given by $ g \mapsto \la(g)$;
  observe that the associated action $\pi^l_0$ is the trivial action, while   the action $\pi^r_0$, which is given by 
  $ g \mapsto \Ad(\la(g)),$  is  ergodic. 
  
  Since $\N(H)$ is a factor   of type $II_1,$ for every $t\in [0,1]$, there exists a projection $p_t$ in $\N(H)$  with $\tau_H (p_t)=t.$ 
   For $g=(h_1, h_2)\in \Ga$ and $t\in [0,1]$, let $u_{g}^{(t)}\in \N(\Ga)$ be defined by
 $$ u_g^{(t)}= \la(h_1) \otimes p_t + 1 \otimes (1-p_t).$$
  Then $u_{g}^{(t)}$ is unitary and 
  $$u^{(t)}: \Ga\to \UU(\N), \, g\mapsto u_{g}^{(t)}$$
   is a group homomorphism.
 For $g\in \Ga,$ let  $\pi_t(g)$ denote the isometry of  $L_p(\N)$  with Yeadon decomposition $(\la(g) u_{g^{-1}}^{(t)}, \Ad(u_g^{(t)})),$ that is,
  $$ \pi_t(g) x= \la(g)u_{g^{-1}}^{(t)}\Ad(u_g^{(t)})(x)=  \la(g) x  u_{g^{-1}}^{(t)} \tout  x\in L_p(\N).$$
 Then $\pi_t: \Ga\to \OO^+(L_p(\N))$ is a group homomorphism.
 
  We claim that $\pi_0$ is not locally rigid. For this, it suffices to show (see Proposition~\ref{Pro-Mazur})
  that  $\pi_0$ is not locally rigid when viewed as homomorphism with values in 
  the $\N$-unitary group $\OO^+_{\N}(L_2(\N)$.
 
  For $g=(h_1, h_2)\in \Ga$ and $t\in [0,1],$ we have
  $$
  \Vert u_g^{(t)}- 1\Vert_2^2= 2(1- \tau (u_g^{(t)}))= 2t(1-\tau_H(\la(h_1)) \leq 2t
  $$
 and hence  $\lim_{t\to 0}\Vert u_g^{(t)}- 1\Vert_2=0.$ 
 It follows that  $\lim_{t\to 0}\pi_t(g)=\pi_0(g)$ in $\OO^+_{\N}(L_2(\N))$ for every $g\in \Ga$.
 
  Assume, by contradiction,  that $\pi_0$ is locally rigid.
  Then, by Proposition~\ref{Pro-Mazur2}, 
  for $t>0$ sufficiently small,  $ \Ad(u_g^{(t)})$ is conjugate to the trivial automorphism $x\mapsto x$ in $\Aut (\N)$
  and hence $ \Ad(u_g^{(t)})$ is  the trivial automorphism for every $g\in \Ga.$ This is a contradiction, as 
  $u_g^{(t)}$ is not a scalar multiple of the identity for $g=(h_1,h_2)$ with $h_1\neq e.$

 Similarly, one can show that the embedding $\rho_0:\Ga \to \OO^+(L_p(\N)),$ given by 
 $$ \rho_0(g): x \mapsto x\la(g^{-1})= \la(g^{-1}) \Ad(\la(g))(x),$$
 is not locally rigid; here,  it is 
  the associated action $\rho^l_0$ which  is  ergodic, while $\rho^r_0$ is the trivial action.
\end{example} 

\begin{example} 
\label{Exa3}
 The conclusion of Theorem~\ref{Theo1} does not hold in general  for actions by isometries on the classical 
 (commutative) $L_p$-spaces. We give a counter-example
 for actions on the sequence space $\ell_p$ (more involved counter-examples
 can be found for actions on the space $L_p[0,1]$).
 
 Let    $\Ga$ be an arbitrary group which is not a torsion group; thus $\Ga$  contains
  a subgroup $A$ isomorphic to $\ZZ.$
 There exists a family $(\chi_t)_{t\in [0, 1]}$
 of unitary characters $\chi_t$ of  $A$ with $\chi_t\neq 1$ for $t\neq 0,$
 $\chi_0= 1$  and such that  $\lim_{t\to 0} \chi_t (a)=1$ for all $a\in A.$
 Choose a set of representatives $X$ for the left cosets of $\Ga$ modulo $A$
 with $e\in X;$ so,
 $\Ga= XA.$ Let $c: \Ga\times X\to A$ be the cocycle defined
 by 
 $$
 g x \in Xc(g, x) \tout  g\in\Ga,\, x\in X.
 $$
 We transfer the natural  $\Ga$-action on $\Ga/A$  to an action $(g, x)\mapsto g(x)$ of $\Ga$ on $X\cong\Ga/A$, by setting
 $$
 g(x)= gx c(g, x)^{-1} \tout  g\in\Ga,\, x\in X.
 $$
   For every $t\in \RRR$  and every $p\in [1,+\infty[,$ 
   the  operator $\pi_{t}(g)$  on 
$\ell_p(X)$, defined by 
$$
\pi_{t}(g)f(x)=\chi_t(c(g^{-1},x))f(g^{-1}(x)) \tout  f\in \ell_p(X), x\in X,
$$
is an isometry and $\pi_t: \Ga\to \OO(\ell_p(X))$ is a homomorphism.
(For $p=2,$ $\pi_t$ is the  unitary representation of $\Ga$ induced by the character $\chi_t$
of $A.$  Observe also that $\pi_0$ is the quasi-regular representation of $\Ga$
in $\ell_p(X)\cong \ell_p(\Ga/A).$ ) 

 We have $\lim_{t\to 0}\pi_t(g)= \pi_0(g)$ 
in $\OO(\ell_p(X)),$ for every $g\in \Ga.$ Indeed,
by linearity and density,  it suffices to check
that 
$$\lim_{t\to 0}\Vert \pi_t(g) \delta_x -\pi_0(g)\delta_x \Vert_p=0,$$
 where $\delta_x$ is the Dirac function at $x\in X.$
This is  the case, since
$$
\Vert \pi_t(g) \delta_x -\pi_0(g)\delta_x \Vert_p= |\chi_t(c(g^{-1},g(x))-1|.
$$

Let  $t\neq 0$  and $p\neq 2.$ We claim that  $\pi_t$ does not belong to the $\OO(\ell_p(X))$-orbit 
of $\pi_0.$ 
  Indeed, assume, by contradiction, that there exists  $U=U_t$ in $\OO(\ell_p(X))$ such that
 $$\pi_t(g)= U\pi_0(g) U^{-1}\tout g\in \Ga.$$
 By Banach characterization  of the isometries of $\ell_p(X)$ from \cite[Chap. XI]{Banach},
 there exists a function $\alpha: X\to \mathbf S^1$ and a bijective mapping $\varphi: X\to X$
 such that   
 $$
(Uf)(x)=\alpha(x)f(\varphi (x)) \tout  f\in \ell_p(X),x\in X.
$$
 One computes that $U\pi_t(g^{-1}) U^{-1}=\pi_0(g^{-1}) $ amounts to the equation
  \begin{align*}
& \alpha(x)\alpha(\varphi^{-1} g \varphi(x))^{-1} \chi_t(c(g, \varphi(x)))f( \varphi^{-1} g \varphi (x))= f(g(x)),
  \end{align*}
 for all $ f\in \ell_p(X)$ and $x\in X.$
 It follows from this that $\varphi^{-1} g \varphi = g$ on $X$ and, consequently, 
 $$
\alpha(x)\alpha( g(x))^{-1}\chi_t (c(g, \varphi(x))=1 \tout g\in\Ga, \, x\in X.\qquad \qquad (*)
 $$
 Let  $x=\varphi^{-1}(e)$ and $g\in A.$ Then, $c(g, e)= g$ and $g(e)=e.$
 Hence, we have  
 $$
 g(x)= g(\varphi^{-1}(e))= \varphi^{-1}(g(e)) = \varphi^{-1}(e)= x
 $$
 If follows from $(*)$ that $\chi_t (g)=1$ for all $g\in A$   and this is a contradiction.
\end{example}

 \subsection*{Acknowledgments} We thank  Mikael de la Salle and Eric Ricard for helpful  comments.
 We are also grateful to the referee for suggesting us to extend
 our main result, which we originally proved for actions by complete isometries,
 to general actions by isometries.
This work was partially supported by the  ANR (French Agence Nationale de la Recherche)
through the projects Labex Lebesgue (ANR-11-LABX-0020-01) and   GGAA  (ANR-10-BLAN-0116).

\end{document}